\documentclass[11pt]{article}
\usepackage{times}
\usepackage{xcolor}
\usepackage{latexsym,amsfonts}
\usepackage{amssymb,amsthm,upref,amscd}
\usepackage[T1]{fontenc}

\usepackage{amscd}
\usepackage{mathrsfs}
\usepackage[reqno]{amsmath}
\usepackage[numbers,sort&compress]{natbib}
\usepackage{braket}
\newtheorem{theorem}{Theorem}
\newtheorem{lemma}{Lemma}[section]
\newtheorem{proposition}[lemma]{Proposition}
\newtheorem{corollary}[lemma]{Corollary}
\newtheorem{remark}[lemma]{Remark}

\makeatletter \@addtoreset{equation}{section} \makeatother

\newcommand{\dist}{\mathrm{dist}}

\newcommand{\R}{\mathbb{R}}

\newcommand{\e}{\varepsilon}
\newcommand{\rd}{\mathrm{d}}

\DeclareMathOperator{\supp}{supp}

\usepackage{authblk}
 
\title
	{\bf  A New Approach to Solving Singularly Perturbed NLS at Local Potential Maxima}

  \author{Chengxiang Zhang\thanks{  zcx@bnu.edu.cn, supported by NSFC-12371107}}
  \affil{\footnotesize School of Mathematical Sciences, Laboratory of Mathematics and Complex Systems, MOE, Beijing Normal University, 100875 Beijing, People's Republic of China}
  \date{}
\begin{document}
 \maketitle
 \begin{minipage}{14cm}
  {\small {\bf Abstract:}
  This paper presents a new approach for addressing the singularly perturbed nonlinear Schr\"odinger (NLS) equation:
   \[
    -\e^2\Delta v + V(x) v =f(v),\ v>0,\ \lim_{|x|\to \infty} v(x)=0,
   \]
   where 
    $V$ possesses a local maximum point and $f$  satisfies the Berestycki-Lions conditions.
    The key to our approach is the derivation of a refined lower bound on the gradient norm.
  
  
  }
  
  \end{minipage}
 \section{Introduction}
In this paper, we consider the following singularly perturbed nonlinear Schr\"odinger equation
\begin{equation}\label{eq1.1}
  -\e^2\Delta v + V(x) v =f(v),\ v>0,\ \lim_{|x|\to \infty} v(x)=0,
 \end{equation}
 which has been a subject of extensive research for several decades.
 Initiated by Floer and Weinstein in \cite{FW}, they demonstrated that when $f(u)=u^3$, Equation \eqref{eq1.1} admits a positive  concentrating peak solution   for small $\e$
  using the Lyapunov-Schmidt reduction technique. 
  Subsequently, Rabinowitz in \cite{Rabinowitz1992} employed global variational methods to investigate the existence of solutions to this problem more generally.
 Wang in \cite{WX1} further highlighted that such concentrated solutions must localize near critical points of the potential function 
 $V$. Building upon these seminal works, numerous substantial advancements have been made in establishing the existence of concentration solutions that cluster around different categories of critical points of 
 $V$. 
 Notable contributions include studies from \cite{ABC,Byeonjean,Byeonjeanjeantanaka,Byeon-Tanaka,KangWei,DelPino-Felmer1996,DelPino-Felmer1997,DelPino-Felmer2002,Li97} among others.
 
 If the limit equation to \eqref{eq1.1} has certain uniqueness and non-degeneracy properties, the Lyapunov-Schmidt reduction method is an effective method for such localized peak solutions, see for example \cite{WX1,KangWei,ABC}.  However, these favorable conditions are only known for a few special nonlinearities. 
 For more general cases where such uniqueness or non-degeneracy is not guaranteed, variational methods have been extensively developed to identify concentration solutions around critical points of $V$.
 See for example
 \cite{ Byeonjean,Byeonjeanjeantanaka,Byeon-Tanaka, DelPino-Felmer1996,DelPino-Felmer1997,DelPino-Felmer2002}.
 Finding solutions that concentrate near local maxima or saddle points of 
$V$ tends to be more challenging due to the lack of least-energy characteristics associated with these points.
In \cite{DelPino-Felmer2002}, del Pino and Felmer  
introduced a variational reduction technique specifically designed to construct solutions that localize at arbitrary local maxima or saddle points of  $V$.
Nevertheless, this methodology relies on a specially defined negative gradient flow over Nehari's manifold, which restricts its applicability and does not extend naturally to cover the broader Berestycki-Lions conditions  introduced in \cite{BerestyckiLions}. 

In \cite{Byeontanaka,Byeon-Tanaka}, for $f\in C^1$ satisfying   the Berestycki-Lions condition, Byeon and Tanaka managed to obtain  localized solutions that concentrate at general local maxima or saddle points of $V$,  by introducing an additional translational flow.
This idea was further integrated into a deformation argument   within an augmented   space, as Cingolani and Tanaka showed recently in \cite{CingolaniTanaka}.

The well-known Berestycki-Lions condition is noted as follows:
\begin{description}
  \item[(F1)] $f \in C(\mathbb{R},\R )$ and $f(0)=\lim\limits_{s \rightarrow 0^+} f(s) / s=0$.
\item[(F2)] If $N\geq 3$ then $\limsup\limits_{s \rightarrow +\infty} f(s) / s^{\frac{N+2}{N-2}}=0 $; if $N=2$, then  $\limsup\limits_{s \rightarrow +\infty} f(s) / e^{\alpha s^2}=0$
for any $\alpha>0$.
   \item[(F3)] There is $t_0>0$ such that 
     \[
      \frac{V_0}2t_0^2<F(t_0),
      \]
      where $F(s)=\int_0^s f(t) \rd t$ and $V_0$ is the constant appeared in (V2).
\end{description}
To state our result, we also need the following condition on $V$:
\begin{description}
	\item[(V1)] $V\in L_{loc}^\infty(\R^N)$ and $ \inf_{x\in\R^N} V(x) >0 $;
\item[(V2)] There is a bounded domain
  $\Omega\subset {\mathbb{R}}^{N}$  such that $V\in C^1(\overline\Omega)$ and $$V_0:=\max\limits_{x\in \overline\Omega}V(x)>\max\limits_{x\in \partial\Omega}V(x);$$
  \item[(V3)]  For any open neighborhood $\widetilde{O}$ of $\mathcal{V}:=\left\{x \in \Omega \mid V(x)=V_{0}\right\}$, there exists an open set $O \subset \widetilde{O}$ such that
  $$
  \mathcal{V} \subset O \subset \overline{O} \subset \widetilde{O} \cap \Omega \quad \text{and}\quad
  \inf _{x \in \partial O}|\nabla V(x)|>0,
  $$
\end{description}
 Without loss of generality, we assume $\inf_{x\in\R^N} V(x)=1<V_0$ and $0\in \mathcal V$.
We will show the following result.
\begin{theorem}\label{th1.1}
Suppose that   (F1)--(F3),  (V1)--(V3) hold. 
There exists $\varepsilon_0>0$ such that for each $\varepsilon\in(0,\varepsilon_0)$, 
equation \eqref{eq1.1} admits a positive solution $ v_\e $  satisfying 
\begin{enumerate}
  \item $v_\varepsilon>0$  has a global maximum point $x_\e$ satisfying $\displaystyle
\lim _{\varepsilon \to 0} \operatorname{dist}\left(x_{\varepsilon}, \mathcal{V}\right)=0.  
$
\item setting $u_{\varepsilon}(x)=v_{\varepsilon}( \varepsilon x),$ there exist a subsequence $\varepsilon_{j} \to 0$ such that
$$  
\left\|u_{\varepsilon_{j}}-U\left(\cdot-x_{\varepsilon_{j}}  / \varepsilon_{j}\right)\right\|_{\e} \to 0 \quad \text { as } j \to \infty, 
$$
where $U$ is a  positive ground state solution   to
$-\Delta U+ V_0 U =f(U_0)$.

\item 
there exist  $C,c>0$ such that
$$v_\varepsilon(x)\leq C e^{-c\varepsilon^{-1}|x-x_\e| }\quad \text{for}\ \ x\in\mathbb R^N.$$
\end{enumerate}

\end{theorem}

It is worth noting that a similar result was initially established in the work of Byeon and Tanaka \cite{Byeon-Tanaka}. In their study, they required an additional assumption that $f$ belongs to $C^1$  to employ a tail-minimizing operator.
The condition $f\in C^1$ 
  was subsequently relaxed by Cingolani and Tanaka in \cite{CingolaniTanaka} through the introduction of a new method for controlling the tail behavior of the functions involved.
  In the paper \cite{CingolaniTanaka}, while adopting arguments within an augmented function space, Cingolani and Tanaka introduced a technical requirement on 
 $\nabla V$. Although it is feasible to eliminate this condition with certain modifications to their arguments, our primary focus lies in introducing an alternative approach to tackle this issue effectively.

By setting $u(x)=v(\e x)$, we have 
\begin{equation}\label{eq1.2}
  - \Delta u + V(\e x) u =f(u),\ u>0,\ \lim_{|x|\to \infty} u(x)=0.
 \end{equation}
 We note that a solution to \eqref{eq1.2} can be obtained as a critical point of 
\[
  J_{\varepsilon}(u)=\frac{1}{2}\int_{\R^N}(|\nabla u|^2+  V(\varepsilon x)u^2)
  - \int_{\mathbb R^N}F(u), \quad u\in H^1(\R^N).
\]
In the process of identifying a localized critical point, employing a local deformation technique is particularly beneficial. This method relies on gradient estimates of the functional within an annular region surrounding a potential candidate critical point. 
However, when dealing with concentration solutions that are expected to cluster around a local maximum (or saddle) point of the potential function 
$V$, a uniform lower bound independent of $\e$ for the gradient in this annular domain does not generally exist.
Such a lack arises in situations where the barycenters of a sequence of functions diverge from the local maximum (or saddle) point of the potential.
Therefore, the crucial aspect of applying the deformation argument successfully under these conditions involves deriving a more precise lower estimate for the gradient when the barycenter of the function 
$u$ is at a certain distance away from the local maximum point. This refined estimate takes the form:
\begin{equation}\label{gradientestimate}
  \|\nabla J_\e(u)\|\geq  C\varepsilon,
\end{equation} 
where 
$C>0$ is a constant independent of the perturbation parameter 
$\e$.
By ensuring such a lower bound on the gradient norm, we can better control the behavior of the barycenters of functions along the gradient flow and hence facilitate the construction of concentrated solutions through a well-designed deformation argument.
If $u $ belongs to a  bounded  set of $H^2$, then this estimate can be obtained.
To elaborate, suppose that   $x_0$ is a limiting point of barycenters of  the sequence 
$\{u_\e\}$, which satisfies $\nu=\nabla V(x_0)\neq 0$.
Consider the directional derivative of the functional 
$J_\e$
  along the direction $\frac{\partial u_\e}{\partial \nu}$,
\begin{equation}\label{j'}
  \langle\nabla J_\varepsilon(u_\e), \frac{\partial u_\e}{\partial \nu}\rangle=\left.\frac{\rd}{\rd t}\right|_{t=0}J_{\varepsilon}(u_\e(\cdot-t \nu))= \varepsilon|\nu|^2 |u_\e|_2^2+o(\varepsilon).
\end{equation}
This will give a contradiction if $\|\nabla J_\varepsilon(u_\e)\|=o(\e)$.

We would like to emphasize that the gradient estimate presented in \eqref{gradientestimate} was initially derived by
del Pino and Felmer in 
   \cite{DelPino-Felmer2002}.
   They achieved this through a specialized negative gradient flow defined over Nehari's manifold. 
   By imposing some specific conditions on $f$, they successfully demonstrated the $H^2$ regularity along this flow.
Nonetheless, it is noteworthy that their method does not readily extend to encompass the broader Berestycki-Lions type conditions.
This highlights the need for alternative techniques to address more general scenarios under these conditions.

We aim to briefly outline our strategy for deriving \eqref{gradientestimate} under the special scenario where $V$ is bounded and $f$ is bounded in $C^1$.
 In this  case, $J_\varepsilon$ is indeed of class $C^2$. 
 Suppose that a sequence $\{u_\varepsilon\}$ satisfies $\|\nabla J_\e(u_\e)\|  =o(\varepsilon)$. We can then find a corresponding sequence
  $\{w_\e\}$ which belongs to a  bounded  set of $H^2$, such that
\begin{align*}
  \|u_\e-w_\e\|_\e=o(\e), \mbox{ and } \|\nabla J_\e(w_\e)\|=o(\e).
\end{align*}
Notably, if we define  $w_\e:= u_\e -\nabla J(u_\e)$, it follows that $\|u_\e-w_\e\|_\e =\|\nabla J_\e(u_\e)\|=o(\e)$. 
Moreover, due to the smoothness of $J$, we have \begin{align*}
\|\nabla J(w_\e)\|\leq \|\nabla J_\e(u_\e)-\nabla J_\e(w_\e)\|+\|\nabla J(u_\e)\|=\| J_\e''(u_\e)\|\|u_\e-w_\e\|+o(\e)=o(\e).
\end{align*}
Since $w_\e$ solves the equation $-\Delta w_\e +V(\e x) w_\e = f(u_\e)$, we can infer that $\{w_\e\}$  resides in a bounded set of    $H^2$.
Consequently, we can substitute 
$u_\e$
  with 
$w_\e$
  in the expression \eqref{j'} and derive a contradiction, thereby establishing the desired gradient estimate.

With a modification of this idea, we are able to rigorously prove this estimate under the more general assumptions (F1)-(F3) and (V1)-(V3)  within the context of this paper.
The proof involves introducing a prior decay estimate   when  gradient bounds are prescribed. This   estimation technique allows us to obtain the desired  gradient estimate under these broader assumptions. 
Upon comparing our method with those from \cite{Byeontanaka, Byeon-Tanaka} by Byeon and Tanaka, their approach involved using both a gradient flow and a translational flow in their deformation procedure. The translational flow was specifically tailored to reduce the energy of functions as their barycenters deviated from local maxima of the potential function. Meanwhile, Cingolani and Tanaka in \cite{CingolaniTanaka} further developed this idea by incorporating a deformation argument within an augmented functional space.
In contrast, our method builds upon the work of del Pino and Felmer \cite{DelPino-Felmer2002} and focuses solely on employing the gradient flow for conducting the deformation analysis. This streamlined strategy enables us to obtain the desired gradient estimate under more general assumptions (F1)-(F3) and (V1)-(V3).
Additionally, it is worth mentioning that similar ideas have been extended to prove the existence of multiple clustering peak solutions for the nonlinear Schr\"odinger equation with a prescribed 
$L^2$ norm constraint, as demonstrated in the work \cite{zhangzhang3}.

\section{Preliminaries}
\subsection{Limit problem}
Throughout this paper, we employ the notation $\|u\|_p$  to represent the 
$L^p(\R^N)$ norm of
 $u$.

For $m\in (0, V_0]$, consider
\[
L_m(u) = \frac12\|\nabla u\|_2^2 +\frac{m}2 \|u\|_2^2 -\int_{\R^N} F(u) : H^1(\R^N) \to \R.  
\]
  Critical points of $ L_{m} (u) $ correspond to weak solutions of the following nonlinear Schrödinger equation:
\begin{equation}\label{eq2.2}
  - \Delta u+mu=f(u),\quad u\in H^1(\R^N) .
\end{equation} 
 
We have from \cite{BerestyckiLions,Jeanjeantanaka02,Byeonjeanjeantanaka} that
\begin{lemma}
  For $m\in(0, V_0]$, the following statements hold.
\begin{enumerate}
  \item[(i)]  Every solution $u(x)$ to \eqref{eq2.2} satisfies the Pohozaev identity: $P_{m}(u)=0$, where
  $$
  P_{m}(u)=\frac{N-2}{2}|\nabla u|_{2}^{2}+\frac{Nm}{2} |u|_{2}^{2}-N \int_{\mathbb{R}^{N}} F(u) .
  $$
  \item[(ii)]
  The least energy level $E_{m}\triangleq  \inf \set{L_{m}(u) \mid u \in H^1(\R^N)\setminus\{0\}, L_{m}^{\prime}(u)=0}$  is attained by  
   a weak solution of \eqref{eq2.2}. Moreover, there holds
  $$
  E_{m}=\inf \set{L_m(u) \mid u \in H^1(\R^N)\setminus\{0\}, P_m(u)=0}=\inf_{\gamma \in \Gamma(m)} \max_{t \in[0,1]} L_{m}(\gamma(t)),
  $$
where $\Gamma(m)=\set{ \gamma(t) \in C\left([0,1]; H^{1}(\mathbb{R}^{N})\right) \mid \gamma(0)=0, L_m(\gamma(1))<0}$.
\item[(iii)]   $E_m$ is continuous and increasing with respect to $m$.
\end{enumerate}
\end{lemma}
For $m\in(0, V_0]$, set
$$
K_{m} =\Set{u \in H^{1}\left(\mathbb{R}^{N}\right) \backslash\{0\} | L_{m}^{\prime}(u)=0, L_{m}(u) \leq E_{V_0}, u>0, u(0)=\max_{x\in \R^N}u(x) }
$$

Then we have from \cite{Byeontanaka,Byeonjeanjeantanaka}
\begin{lemma}\label{compact}
  For every   $\delta\in(0, V_0)$.
 The set $\cup_{m\in [V_0-\delta, V_0]}K_m $ is compact in $H^1(\mathbb R^N)$. Moreover, there are constants $C,c>0$ such that  for any $U\in \cup_{m\in [V_0-\delta, V_0]}K_m$,
  \[
    U(x)+|\nabla U(x)| \leq C\exp(-c|x|)\mbox{ for all } x\in \R^N.
  \]
\end{lemma}
 
Take $U_0\in K_{V_0}$, we know from \cite{Jeanjeantanaka02} that 
\begin{lemma} \label{lemma2.3}
  \begin{enumerate}
    \item[(i)] When $N\ge3$, 
    \[P_{V_0}(U_0(e^{-\theta}\cdot))=\frac{\rd}{\rd\theta}L_{V_0}(U_0(e^{-\theta}\cdot))>0 \quad \mbox{if}\  \theta<0,\] 
    \[P_{V_0}(U_0(e^{-\theta}\cdot))=\frac{\rd}{\rd\theta}L_{V_0}(U_0(e^{-\theta}\cdot))<0 \quad \mbox{if}\  \theta>0.\] 
    \item[(ii)] When $N=2$,  \[ L_{V_0}(U_0(e^{-\theta}\cdot))\equiv E_{V_0}\quad \mbox{for}\ \theta\in\R.\]   Moreover, there is 
    $\theta_0>0$ and $s_0>1$ such that
    \[P_{V_0}(sU_0(e^{-\theta_0}\cdot))>0,\quad \frac{\rd}{\rd s}L_{V_0}(sU_0(e^{-\theta_0}\cdot))>0\quad \mbox{for}\ s\in [0,1],\]
    \[P_{V_0}(sU_0(e^{-\theta_0}\cdot))<0,\quad \frac{\rd}{\rd s}L_{V_0}(sU_0(e^{-\theta_0}\cdot))<0\quad \mbox{for}\ s\in [1, s_0].\]
  \end{enumerate}
  
\end{lemma}
 
By (V1)-(V3), take a   neighborhood  $O$   of $\mathcal V$  and $\delta_0>0$ sufficiently small such that 
\begin{gather}
0\in \mathcal V\subset O\subset  O^{5\delta_0}\subset \Omega, \inf_{x\in O^{3\delta_0}\setminus O}|\nabla V(x)|>0,\label{equ2.2}\\
\inf_{x\in O^{5\delta_0}}V(x)\geq V_0-\delta_0>1,\quad 
\mbox{and}\quad 2E_{V_0-\delta_0}>E_{V_0},\label{equ2.3}
\end{gather}
where  $O^\delta=\set{x\in \mathbb R^N | \text{dist}(x, O)\leq \delta}$.
We define 
\begin{equation}
 S_0 = \begin{cases}
  \displaystyle  \Set{u(e^{-\theta}\cdot) | u\in K_m, m\in [V_0-\delta_0, V_0],  \theta\in [-\theta_0, \theta_0]}, &N=2,\\
  \displaystyle  \bigcup_{m\in [V_0-\delta_0, V_0]} K_m, & N\geq 3.
  \end{cases}
\end{equation}
Then $S_0$ is compact in $H^1(\R^N)$.

\subsection{Barycenter function}

For given $\varepsilon>0$, we define the space
\begin{equation*}
{H}_{\varepsilon}=\Set{u\in H^1(\mathbb{R}^N) | \int_{\mathbb{R}^N}V(\varepsilon x)u^2\rd x<+\infty}
\end{equation*}
equipped with the norm
\begin{equation*}
\|u\|_{\varepsilon}  := (\|\nabla u\|_2^2+\int_{\mathbb{R}^N}V(\varepsilon x)u^2\rd x)^{\frac{1}{2}}.
\end{equation*}
By the compactness of $S_0$,
 we can find  $R_0>1$ such that for 
each 
$U\in S_0$, there holds
\begin{equation}\label{equa 3.10}
\|U\|_{L^2(B(0,R_0/2))}>\frac{3}{4}\rho_1,\quad  \|U\|_{L^2(\mathbb R^N\setminus B(0,R_0))}<\frac{\rho_1}{8}.
\end{equation}
Take $\phi\in C_0^\infty (O^{4\delta_0};[0,1])$ such that 
\begin{gather*}
   \phi(x) =1 \mbox{ for } |x|\leq \delta_0/2, \quad 
    \phi(x) =0\mbox{ for }|x|\geq \delta_0,\quad 
    |\nabla \phi| \leq 4/\delta_0 \mbox{ in } \R^N.
  \end{gather*}
  For $\e>0$, set $\phi_\e(x)= \phi(\e x)$.
We   set 
\begin{equation}
   S_\e(\Omega) =\Set{   (\phi_\e U) (\cdot - y ) | \ \e y\in \Omega,\  U \in S }.
\end{equation}
 Define
\begin{equation}
  Z_\e =\Set{u\in H_\e | \dist_{H_\e}(u, S_\e(\Omega))< \frac{\rho_1}{16}  }.
 \end{equation}

We will recall the barycenter function  in \cite{zhangzhang2,CingolaniTanaka}, which is a smooth counterpart of that in \cite{Byeontanaka,Byeon-Tanaka}. 
First note that there is $\e_1>0$ such that for $\e\in(0,\e_1)$,
   $u\in Z_\e$,
   there hold
  \[\int_{B(P, R_0)}   u^2 \geq \frac1 2 \rho_1^2\quad  \mbox{for}\quad  P\in    B(y ,R_0/2), 
  \quad \int_{B(P, R_0)}   u^2\leq \frac{1}{16}\rho_1^2 \quad  \mbox{for}\quad  P\notin
    B(y,2R_0).  
  \]
For $u\in H^1(\R^N) $ and $P\in\mathbb R^N$, we define
\begin{equation}
d(u,P)=\psi\left( \int_{B(P, R_0)}   u^2\right),
\end{equation}
with $\psi\in C_0^\infty([0,\infty), [0,1])$ satisfying
\begin{equation*}
\psi(r)=
\begin{cases}
0 \  \ &r\in[0,\frac{1}{16}\rho_1^2],\\
1 \ \ &r\in[\frac{1}{2}\rho_1^2,\infty).
\end{cases}
\end{equation*}
 
We define
\begin{equation}
\Upsilon (u)=\frac{\int_{\R^N}d(u,P)PdP}{\int_{\R^N}d(u,P)dP}\in \R^N.
\end{equation}
 
 Similarly to \cite[Lemma 2.5]{zhangzhang2}, we have 
\begin{lemma}\label{lem3.4}
The following statements hold for $\e\in(0, \e_1)$.
\begin{itemize}
  \item[(i)] If $\|u- (\phi_\e U) (\cdot - y )\|_\e<  {\rho_1}/{16}$ for $y\in \frac1\e \Omega,  U\in S $, we have
  $|\Upsilon (u)-y |\leq 2R_0$,   where use notation $ \frac1\varepsilon \Omega=\set{x\in\R^N|\e x\in \Omega}$.
\item[(ii)] $\Upsilon (u)$ is $C^1$ continuous for each $u\in  Z_\e$. Moreover, there exists a constant $D_1>0$   such that
    $$\sup_{u\in  Z}\|\Upsilon '(u)\|_{\mathcal L(H_\e,\R^N)}\leq D_1.$$
\item[(iii)] If $u,v\in    Z_\e$ satisfy for some  $h\in\mathbb R^N$ that
$$v(x-h)=u(x)\ \ \ \text{in}\ B(\Upsilon (u),4R_0),$$
then $\Upsilon (v)=\Upsilon (u)-h$.
\item[(iv)] $\Upsilon'(u)v=0$ if $\supp v\subset \R^N \setminus B(\Upsilon(u), 4R_0)$.
\end{itemize}
\end{lemma}
 \subsection{Penalized functional}

 For
 $\rho\leq \frac{1}{16}\rho_1$, $\delta\in[\delta_0, 3\delta_0]$, set 
  \begin{equation}
  Z_\e(\rho,\delta)= \Set{u\in H_\varepsilon|\dist_{H_\varepsilon}(u, S_\e(\Omega))<  {\rho } ,\  \dist(\varepsilon\Upsilon(u), O)<\delta     }.
  \end{equation}
  
 \begin{remark}
  Let $\rho<\rho'$ and $\delta<\delta'$. Then,  for $\varepsilon$ sufficiently small,
   \[
   \dist_{H_\varepsilon}(\partial Z(\rho',\delta'), Z(\rho,\delta)) \ge\min\{\rho'-\rho, {\rho_1} \}.  
   \]
   In fact, if $\dist_{H_\varepsilon}(u, S_\e(\Omega))=  \rho' $, then
    \[\dist_{H_\varepsilon}(u, Z_\e(\rho,\delta)) \ge\rho'-\rho.\] 
   If $\dist_{H_\varepsilon}(u, S_\e(\Omega)) < \rho'$ and $ \dist(\e\Upsilon(u), O)=\delta'$, then by \eqref{equa 3.10} and Lemma \ref{lem3.4},
   \[\liminf_{\e\to 0}\dist_{H_\varepsilon}(u, Z_\e(\rho,\delta))\ge 
   2\inf_{U\in S_0} \|U\|_{L^2(B(0, R_0/2))} - \frac{\rho_1}{16}-\frac{\rho_1}8 >  \rho_1 .\]
 \end{remark}

Note  that there is $D_2>0$ independent of $\e$ such that 
 \begin{equation}
 \|u\|\leq  \|u\|_\e\leq D_2\quad \mbox{for any } u\in Z_\e(\rho_1,3\delta_0).
 \end{equation}
 The following result follows directly from the Sobolev inequality when $N\geq 3$ 
 and \cite[Lemma 1]{Byeonjeanjeantanaka} when $N=2$.
\begin{lemma}\label{lemma26}
  Let $f$ satisfies (F1)-(F3). Assume 
  $w\in H^1(\R^N)$   weakly satisfies $\|w\|\leq D_2$ and
  \[
   \int_{\R^N} |\nabla w |^2  +  w^2  \leq   \int_{\R^N} |f(w) w|,
  \]
  then either  there is $\rho_2>0$ independent of $w$ such that $|w|_2\geq \rho_2$ or $w=0$.
\end{lemma}
  We will find a solution of \eqref{eq1.2} in  $ Z_\e(\rho_1,3\delta_0)$.
 Next, we note that we can assume $f$ is bounded by similar arguments of \cite{Byeontanaka}.
  By (F1) and (F2), there is a sufficiently small $\alpha_0$ and a $C_0>0$ such that any 
  solution $u$ of \eqref{eq1.2} satisfies 
  \begin{equation}\label{eq2.12}
    \begin{aligned}
      -\Delta |u| + |u|& \leq \alpha_0 |u|+ C_0|u|^{\frac{N+2}{N-2}} \quad  & \mbox{if } N\geq 3,\\
      -\Delta |u| + |u| &\leq \alpha_0|u| +C_0(e^{\alpha_0|u|^2}-1) \quad  & \mbox{if } N=2.
    \end{aligned}
  \end{equation} 
If $u \in Z_\e(\rho_1,3\delta_0)$, then $\|u\|_\e$ is bounded by a constant independent of $\e$. By the  
elliptic estimates, there is a constant $K>0$ such that $\|u\|_{L^\infty(\R^N)}\leq K$. Therefore, we can 
set  $\tilde f(t) = f(t)$ if $t\leq 2K$,   $\tilde f(t)= f(2K)$ if $t\geq 2K$. Then there still holds
$\tilde f(t)\leq \alpha_0 t+ C_0(t^{\frac{N+2}{N-2}})$ if $N\geq 3$, and 
$\tilde f(t)\leq \alpha_0t + C_0(e^{\alpha_0t^2}-1)$ if $N=2$. 
As a result, for small $\e$, any solution in $Z_\e(\rho_1,3\delta_0)$ of \eqref{eq1.2} with $f$ replaced by $\tilde f$
satisfies the original equation. From now on, we can assume without loss of generality that $f$ satisfies further that
there is $\tilde K>0$ such that 
\begin{equation}\label{eq2.11}
|f(t)| \leq \tilde K \quad \mbox{for any } t\geq 0.
\end{equation}

    Take $\chi \in C^\infty(\mathbb R^N;[0,1])$ such that 
    \[\mbox{$\chi=1$ in $  \mathbb R^N\setminus  B(0,  2)$, $\chi =0$ in $ B(0,   1)$ and $|\nabla  {\chi}|\leq 2  $.}\]
    Setting $\chi_{\e, u}(x)= \chi \left(\e^\frac12 (x-\Upsilon(u)) \right)$, we 
    define
    $$\Phi_\varepsilon(u)=\left(\e^{-\frac12} \int_{\R^{N} }\chi_{\e, u} u^2 \mathrm dx-1\right)_+^2.
    $$
  Then,   easily we can check that
  \begin{lemma}\label{Phi'}
  There is  $C_0>0$ independent of $\varepsilon$ such that  for $u\in  Z_\e(\rho_1, 3\delta_0)$ and any $v\in H_\varepsilon$,
  \[\begin{aligned}
    \Bigg| \Phi_{\varepsilon}'(u)v- 4\Phi_{\varepsilon}(u)^{\frac12} \e^{-\frac12}\int_{\R^{N} }\chi_{\e, u} uv     \Bigg| 
    \leq  C_0\Phi_{\varepsilon}(u)^{\frac12}  \|v\|_\e  \int_{ \R^N\setminus B(\Upsilon (u), \e^{-\frac12})}  u^2.
  \end{aligned}\]
  In addition if $\supp v \subset \R^N \setminus B(\Upsilon(u), 4R_0)$ then 
  \[
    \Phi_{\varepsilon}'(u)v=4\Phi_{\varepsilon}(u)^{\frac12} \e^{-\frac12}\int_{\R^{N} }\chi_{\e, u} uv    
  \]
  \end{lemma}
  Define the functional:
  \begin{equation}\label{2.3}
  \Gamma_{\varepsilon}(u)=\frac{1}{2}\int_{\R^N}(|\nabla u|^2+  V(\varepsilon x)u^2)
  - \int_{\mathbb R^N}F(u)+\Phi_{ \varepsilon}(u)  ,\quad u\in  Z_\e(\rho_1, 3\delta_0).
  \end{equation}
  We note that   $\Gamma_{\varepsilon}$ is well-defined and is of  class $C^1$ on $Z_\e(\rho_1, 3\delta_0)$.

  \begin{lemma}\label{lem3.12}
   Let $L>0$.
    If $u\in  Z_\e (\rho_1, 3\delta_0)$   satisfies
      $ \Gamma_{\varepsilon}(u) < L$,   
     then  there is  some constant $C(L)>0$ independent of $\varepsilon$ such that
     \begin{equation*}
      \begin{gathered}
          \Phi_{ \varepsilon}(u)+\|f(u)u\|_{L^1(\R^N)}+\|F(u)\|_{L^1(\R^N)}+ 
           \e^{-\frac12}\int_{\R^{N} }\chi_{\e, u}  u^2 \leq C(L).
      \end{gathered}
     \end{equation*}
  \end{lemma}
  \begin{proof}
    Clearly,
    $\|u\|_{\varepsilon }\leq C$ for some $C>0$ independent of  $L, \varepsilon$. Hence, 
    by \eqref{eq2.11},
    $\|f(u)u\|_{L^1(\R^N)}+\|F(u)\|_{L^1(\R^N)}\leq C$.
    Then,   we have 
     \[
      \Phi_{ \varepsilon}(u)  
      \leq \Gamma_{\varepsilon }(u)-\frac12 \|u\|_\e^2 +\int_{\R^N}F (u)  \leq C(L),
     \quad
     \e^{-\frac12}\int_{\R^{N} }\chi_{\e, u}  u^2 \leq  \Phi_{\e}(u)^\frac12+ 1 \leq C(L),
  \]
  which completes the proof.
  \end{proof}
 
 \subsection{A prior decay estimate}
 First we need the following lemma.
 \begin{lemma}\label{lemma 2.9}
  Let $\theta>1$, $b\geq 0$, $R_1, R>0$ be such that $R>R_1+1$. 
  Assume $Q(r)$ is a nonincreasing function in $[R_1, R]$  satisfying
  \[
    Q(r)\leq \theta^{-1} Q(r-1) +b \quad \mbox{for } r\in [R_1+1, R].
  \]
  Then 
  \[
  Q(R)\leq \theta^{R_1+1}Q(R_1) e^{-R\ln\theta}+\frac{\theta b}{ \theta-1}.  
  \]
 \end{lemma}
 \begin{proof}
 By the assumptions, we can get the conclusion from 
 \[
 (Q(R)- \frac{\theta b}{ \theta-1})^+ \leq \theta^{-1} (Q(R-1)- \frac{\theta b}{ \theta-1})^+\leq \theta^{-\lfloor R-R_1\rfloor} (Q(R_1)- \frac{\theta b}{ \theta-1}). \qedhere
 \]
 \end{proof}
 \begin{proposition}\label{lemdecay} 
 There is $\rho_0\in (0, \min\{\rho_1, \rho_2\})$ such that the following statement hold.  
  If  $u\in  Z_\e(\rho_0, 3\delta_0)$   satisfy
  \begin{gather*}
        \|\Gamma_{\e}' (u ) \|_{H_\e^{-1}}\leq b_\e\quad \mbox{for some   $b_\e\geq 0$},
  \end{gather*}
    then there are positive constants $C,c$   independent of $\e, b_\e, u$ such that for each $R \in (0,\e^{-\frac12})$,
  \[\int_{\R^N\setminus   B(\Upsilon(u), R)    }\left(|\nabla u|^2+  u^2  \right) \leq C(b_\e^2+e^{-cR}).\]  
 \end{proposition}
 \begin{proof}
  We only need to prove the result for large $R$.
  First note that,   by Lemma \ref{lem3.4} and the compactness of $S_0$, for each given $\rho_0\in (0, \min\{\rho_1, \rho_2\})$,  there is $R_1>4R_0$ such that 
  \begin{equation}\label{eq2.15}
    \sup_{u\in  Z_\e(\rho_0, 3\delta_0)}  \int_{\R^N\setminus   B(\Upsilon(u), R_1)    }  \left(|\nabla u|^2+  u^2  \right)  \leq 2 \rho_0^2.
  \end{equation}
  For  $R\in[R_1+1, \varepsilon^{-\frac12}]$ and $r \in [R_1+1, R] $, we take $\psi_r \in C^1(\mathbb R^N,[0,1])$
  such that $| \nabla \psi_r |\leq 2$ and
 \[
 \psi_r (x)=
 \left\{\begin{aligned}
 &0 & &{\rm if} & &x\in B(\Upsilon (u),   r-1),\\
 &1 & &{\rm if} & &x\in \R^N\setminus   B(\Upsilon (u),   r),
 \end{aligned}\right.
 \]
By Lemma \ref{Phi'} and $\supp (\psi_r u) \subset \R^N\setminus B(\Upsilon(u), 4R_0)$, we have 
\[
\Phi_\e'(u)(\psi_r u)= 4 \Phi_{\varepsilon}(u)^{\frac12} \e^{-\frac12}\int_{\R^{N} }\chi_{\e, u} \psi_r u^2\geq 0. 
\]
Then   
we have
\begin{equation}\label{eq4.2} \begin{aligned}
    \Gamma_{\e}' (u )(\psi_r u)\geq&
  \int_{\R^N}\psi_r (|\nabla u|^2+ {V}_\varepsilon u^2-f(u)u)  
  +\int_{\R^N} u\nabla \psi_r \nabla u\\
  \geq & \int_{\R^N}\psi_r (|\nabla u|^2+ {V}_\varepsilon u^2-f(u)u) -\int_{\supp |\nabla \psi_r|} ( |\nabla u|^2 +  u^2)
\end{aligned}
\end{equation}

On the other hand,
we have 
\begin{equation}\label{eq4.2'}\begin{aligned}
  \Gamma_{\e}' (u )(\psi_r u)\leq& b_\e \|\psi_r u\|_\e\\
  \leq&  
b_\e^2 + \frac{1}{4} \int_{\R^N}  (|\nabla(\psi_r u)|^2+ {V}_\varepsilon (\psi_r u)^2)
  \\
  \leq &  
  b_\e^2 +
  \frac{1}{4}\int_{\R^N}\psi_r  (|\nabla u|^2+ {V}_\varepsilon u^2 )  
+  2\int_{\supp |\nabla \psi_r|} ( |\nabla u|^2 +  u^2).
\end{aligned}
\end{equation}

By \eqref{eq2.11} and (F1), for $p\in (2, \frac{2N}{N-2})$ if $N\ge 3$ and $p=3$ if $N=2$, there is $C_p>0$ such that 
 \[
 f(u)u \leq \frac{1}{4} u^2 + C_p |u|^p.
 \]
 Then
 setting 
 \[
 Q(r)=\int_{\R^N\setminus B(\Upsilon(u), r)} |\nabla u|^2+   u^2,
 \]
   from  \eqref{eq4.2}, \eqref{eq4.2'}, and  the Sobolev inequality, we conclude
\[\begin{aligned}
  b_\e^2 \geq & \frac12 Q(r) -
  C_p\int_{\mathbb R^N}\psi_r |u|^p  -3(Q(r-1)-Q(r))\\
   \geq &  \frac72Q(r) -C_NC_p (Q(r-1))^p -3Q(r-1),
\end{aligned}
\]
where $ C_{ N}>0$ is a constant  depending only on   $N$.  By \eqref{eq2.15}, $Q(r-1)\leq \sqrt{2} \rho_0$.
Taking $\rho_0>0$ small such that
$C_NC_p(\sqrt2\rho_0)^{p-1}<1/2$, we can  complete the proof by Lemma \ref{lemma 2.9}.
 \end{proof}
 
\begin{corollary}\label{cor2.11}
  Under the assumptions of Proposition \ref{lemdecay}, if $\e^{-1/4}b_\e\to 0$ as $\e\to 0$, then 
  $\Phi_\e(u)=0$ and $\Phi_\e'(u)=0$ for each small $\e$.
\end{corollary}
\section{Gradient estimates}
The following $\e$-dependent concentration compactness result gives a uniform 
lower gradient estimate in $ Z_\e(\rho_0,3\delta_0 )\setminus Z_\e(\frac13\rho_0, 3\delta_0 )$.
\begin{proposition}\label{concompact}
  Suppose
  $\varepsilon_n\to0, u_n\in  Z_\e(\rho_0,3\delta_0 )$   satisfying
\begin{equation}\label{13}
 \limsup_{n\to\infty}\Gamma_{\varepsilon_{n}} (u_{n} ) \leq   E_{V_0},\quad
 \lim_{n\to\infty} \|\Gamma_{\varepsilon_{n}}  (u_{n} ) \|_{H^{-1}_\e}  = 0.
\end{equation}
Then there exist
$U\in S_0$ and 
$ z_{n} \in \frac{1}{\varepsilon_n} \Omega$ such that as $n \to \infty$ (after extracting a subsequence if necessary)
\begin{equation*}
  \|u_{n}-   (\phi_{\e_n}U)  (\cdot-z_{n} )\|_{ {\varepsilon_n}}\to 0.
\end{equation*}
\end{proposition}
The proof of Proposition \ref{concompact} is standard, and will be given in the Appendix.
For $c\in\R$, we set 
\[
\Gamma_\e^{c}:=\set{u\in H_\e | \Gamma_\e(u)\leq c}.  
\]
Let 
$c_\e$ be a sequence to be determined such that $c_\e\to E_{V_0}$ as $\e\to 0$.
By Proposition \ref{concompact},
we have 
\begin{corollary}\label{cor3.2}
  There is $\nu_1, \e_1>0$ such that for any $\e\in(0,\e_1)$, 
 \[ \|\Gamma_\e'(u)\|_{H^{-1}_\e}\geq 2\nu_1\quad \mbox{provided that } u\in \left(Z_\e(\rho_0,3\delta_0 )\setminus Z_\e(\frac13\rho_0, 3\delta_0 ) \right)\cap \Gamma_\e^{c_\e}.\]
\end{corollary}
The main aim in this section is to give the following   gradient estimate in $Z_\e(\rho_0,3\delta_0 )\setminus Z_\e(\rho_0, \delta_0 )$.
\begin{proposition}\label{gradient estimate 2}
  There exist $\nu_2, \e_2>0$ such that for $\e\in (0, \e_2)$,
  \[
      \|\Gamma_\e'(u)\|_{H^{-1}_\e}\geq 2\nu_2\e\quad \mbox{provided that } u\in \left(Z_\e(\rho_0,3\delta_0 )\setminus Z_\e(\rho_0, \delta_0)\right)\cap \Gamma_\e^{c_\e}.  
  \]
\end{proposition}
\begin{proof}
Assume by contradiction that    there is a sequence $u_\e\in \left(Z_\e(\rho_0,3\delta_0 )\setminus Z_\e( \rho_0, \delta_0 )\right)\cap \Gamma_\e^{c_\e}$ such that 
as $\e\to 0$, $\|\Gamma_\e'(u_\e)\|_{H^{-1}_\e}=o_\e(\e)$.
By Proposition \ref{concompact}, up to a subsequence,  
there exist
$U\in S_0$ and 
$ z_\e \in \frac{1}{\varepsilon } \Omega$ such that  
  $\|u_\e-   (\phi_{\e }U)  (\cdot-z_\e )\|_{ {\varepsilon }}\to 0.$
 By Lemma \ref{lem3.4} (i), $|z_\e-\Upsilon(u_\e)|\leq 2R_0$. Hence, $\e z_\e \to z_0\in \overline{O^{3\delta_0}\setminus O^{\delta_0}}$.
By $\inf_{O^{3\delta_0}\setminus O^{\delta_0}}|\nabla V|>0$, we assume without loss of generality that $\frac{\partial V}{\partial x_1}(z_0)>0$ and 
$\frac{\partial V}{\partial x_1}(\e x)>0$
 for $x\in B(z_\e, 2\e^{-1/2})$.
 By Corollary \ref{cor2.11} and Proposition \ref{lemdecay}, we have 
$\Phi_\e(u_\e)=0$, $\Phi_\e'(u_\e)=0$, and 
 \begin{equation}\label{eq3.3}
  \int_{\R^N\setminus B(z_\e, 2\e^{-1/2})} |\nabla u_\e|^2 + |u_\e|^2 = o_\e(\e).
  \end{equation}
  By (F1) and  \eqref{eq2.11}, for some $C>0$, $|f(t)|\leq Ct$. Then 
  \begin{equation}\label{eq3.3f}
    \int_{\R^N\setminus B(z_\e, 2\e^{-1/2})} |f(u_\e)u_\e| + |F(u_\e)| = o_\e(\e).
    \end{equation}
 We next consider $w_\e\in H_\e$ defined by the unique solution to the following equation 
 \begin{equation}\label{eq3.2}
  -\Delta w_\e +V_\e w_\e =f(u_\e), \quad w_\e\in H_\e.
 \end{equation}
 We note that \eqref{eq3.2} is solvable by Riesz representation. 
 Then
 \[
 \|w_\e\|_\e^2\leq \int_{\R^N} f(u_\e) w_\e \leq C\|u_\e\|_\e \|w_\e\|_\e.
 \]
 Hence $\|w_\e\|_\e$ is bounded. Moreover, we have
 \begin{align*}
 o_\e(\e)\|u_\e-w_\e\|_\e &\geq \Gamma'(u_\e) (u_\e-w_\e)\\
 & =(u_\e, u_\e -w_\e) -\int_{\R^N} f(u_\e)(u_\e-w_\e)=\|u_\e-w_\e\|_\e^2.
 \end{align*}
 Therefore, 
 \begin{equation}\label{eq3.5}
  \|u_\e-w_\e\|_\e=o_\e(\e).
 \end{equation}
 Hence,
 \begin{equation}\label{eq3.3'}
  \int_{\R^N\setminus B(z_\e, \e^{-1/2})} |\nabla w_\e|^2 + |w_\e|^2 = o_\e(\e).
  \end{equation}
 By \eqref{eq2.11} and the elliptic estimates, we know that $w_\e\in H^2_{loc}(\R^N)$ and there is a constant $C>0$ such that
 \begin{equation}\label{eq3.6}
 \|\Delta w_\e\|_{L^2(z_\e,  3\e^{-\frac12})}\leq  \|V_\e w_\e\|_{L^2(z_\e,  3\e^{-\frac12})} +\|f(u_\e)\|_{L^2(z_\e,  3\e^{-\frac12})}\le C.
 \end{equation}
Taking $\psi_\e\in C_0^\infty(\R^N)$ such that $0\le\psi_\e\leq1$, $|\nabla \psi_\e|\leq 2\e^{\frac12}$ in $\R^N$, and 
$\psi_\e=1$ in $B(y_\e, 2\e^{-1/2})$, $\psi_\e=0$ in $\R^N\setminus B(y_\e, 3\e^{-1/2})$.
Multiplying \eqref{eq3.2} by $\frac{\partial (\psi_\e u_\e)}{\partial x_1}$ and integrating, we obtain 
\[
\int_{\R^N}(-\Delta w_\e)  \frac{\partial (\psi_\e u_\e)}{\partial x_1} +\int_{\R^N} V_\e w_\e \frac{\partial (\psi_\e u_\e)}{\partial x_1} = \int_{\R^N} f(u_\e) \frac{\partial (\psi_\e u_\e)}{\partial x_1}.
\]
By \eqref{eq3.3}, \eqref{eq3.3f}, \eqref{eq3.5}, \eqref{eq3.3'}, and \eqref{eq3.6}, we have
\begin{align*}
  &\int_{\R^N}(-\Delta w_\e)  \frac{\partial (\psi_\e u_\e)}{\partial x_1} =\int_{\R^N}(-\Delta w_\e)  \frac{\partial (\psi_\e w_\e)}{\partial x_1} +o_\e(\e)\\
  =&\int_{\R^N}\left\{\frac12 \frac{\partial(\psi_\e|\nabla  w_\e|^2)}{\partial x_1}+ \frac12 |\nabla  w_\e|^2 \frac{\partial \psi_\e}{\partial x_1} +  w_\e \nabla  w_\e \nabla \frac{\partial \psi_\e}{\partial x_1} +\nabla  w_\e \nabla \psi_\e \frac{\partial  w_\e}{\partial x_1}\right\}+o_\e(\e)\\
  =&\int_{\R^N\setminus B(y_\e, 2\e^{-1/2})}\left\{ \frac12 |\nabla  w_\e|^2 \frac{\partial \psi_\e}{\partial x_1} +  w_\e \nabla  w_\e \nabla \frac{\partial \psi_\e}{\partial x_1} +\nabla  w_\e \nabla \psi_\e \frac{\partial  w_\e}{\partial x_1}\right\}+o_\e(\e)=o_\e(\e),
\end{align*}
\begin{align*}
  \int_{\R^N} f(u_\e) \frac{\partial (\psi_\e u_\e)}{\partial x_1}=\int_{\R^N}\left\{  \frac{\partial(\psi_\e F(u_\e))}{\partial x_1}+
    \frac{\partial \psi_\e}{\partial x_1} [ f(u_\e)u_\e - F(u_\e)] 
    \right\}=o_\e(\e),
\end{align*}
and 
\begin{align*}
   \int_{\R^N} V_\e w_\e \frac{\partial (\psi_\e u_\e)}{\partial x_1}=&\int_{\R^N} V_\e w_\e \frac{\partial (\psi_\e w_\e)}{\partial x_1} +o_\e(\e)\\
  =&\frac12\int_{\R^N} \left\{
    \frac{\partial ( V_\e\psi_\e  w_\e^2)}{\partial x_1}+  
    {V}_\e \frac{\partial\psi_\e}{\partial x_1}  w_\e^2 -\frac{\partial {V}_\e}{\partial x_1}\psi_\e  w_\e^2\right\} +o_\e(\e)\\
    =&-\frac12\int_{\R^N}  \frac{\partial {V}_\e}{\partial x_1}\psi_\e  w_\e^2 +o_\e(\e).
\end{align*}
Then we obtain
\[
  -\frac12\int_{\R^N}  \frac{\partial {V}(\e x)}{\partial x_1}\psi_\e  w_\e^2 =o_\e(1).
\]
Taking limits as $\e\to 0$, we have 
\begin{align*}
   \frac{\partial V(z_0)}{\partial x_1} |U|_2^2
  = &
  \lim_{\varepsilon\to 0}  \int_{B(z_\e, 2\e^{-1/2})} \frac{\partial {V}(\e x)}{\partial x_1}    u_\e^2\\
  =&\lim_{\varepsilon\to 0}  \int_{B(z_\e, 2\e^{-1/2})}\frac{\partial {V}(\e x)}{\partial x_1}   w_\e^2
  \leq \lim_{\varepsilon\to 0}  \int_{\R^N}\frac{\partial {V}(\e x)}{\partial x_1}\psi_\e  w_\e^2 =0.
\end{align*}
This is a contradiction.
\end{proof}
\begin{remark}
Since 
\[Z_\e(\frac23\rho_0,2\delta_0 )\setminus Z_\e( \frac13\rho_0, \delta_0 )\subset (Z_\e(\rho_0,3\delta_0 )\setminus Z_\e( \rho_0, \delta_0 ))\cup 
(Z_\e(\rho_0, 3\delta_0)\setminus Z_\e(\frac13\rho_0, 3\delta_0)),\] we have
  \[
    \|\Gamma_\e'(u)\|_{H^{-1}_\e}\geq 2\min\set{\nu_1, \nu_2\e}\quad \mbox{provided that } u\in \left(Z_\e(\frac23\rho_0,2\delta_0 )\setminus Z_\e( \frac13\rho_0, \delta_0 )\right)\cap \Gamma_\e^{c_\e}.  
\]
\end{remark}
\section{Proof of the main theorem}
In this section we show the main theorem. First we determine
the sequence $c_\e\to E_{V_0}$. Let $U_0\in K_{V_0}$ be as in Lemma \ref{lemma2.3}.
Let $(p, s) \in   O^{ \delta_0} \times [-1,1]$.
For $\theta_1\in(0, 1/2)$, set 
\[
  \theta(s) :=\begin{cases}
    2(1-\theta_1) s + 2-\theta_1 & s\in [-1, -1/2],\\
    1,& s\in [-1/2, 1/2],\\
    2\theta_1 s+ 1-\theta_1, & s\in[1/2, 1].
  \end{cases}  
  \] 
Define
$$\gamma_{0\e}({p},{s}) (x):= \begin{cases} \displaystyle (\phi_\e U_0)( e^{-\theta_1 s}  (x-\frac{p}\e) ), & N\geq 3,\\
    \displaystyle\theta(s)  (\phi_\e U_0)( e^{-2\theta_1s}  (x-\frac{p}\e)  ), & N=2,
\end{cases}
  $$
  and 
  \[
    c_\e:=\max_{(p,s)\in   O^{ \delta_0} \times [-1,1]}\Gamma_\e (\gamma_{0\e}({p},{s})).
  \]
We have 
\begin{lemma}\label{lem4.1}
  There are $\theta_1, \nu_3, \e_3>0$ such that the following statements hold for $\e\in(0, \e_3)$.
  \begin{description}
    \item[(i)]  $\gamma_{0\e}({p},{s}) \in Z_{\e}(\frac13 \rho_0, \delta_0)$ for each $(p, s) \in   O^{ \delta_0} \times [-1,1]$.
    \item[(ii)] There is $R_1>4R_0$ such that $|p-\e\Upsilon(\gamma_{0\e}({p},{s}))|\leq R_1\e$ for each $(p, s) \in   O^{ \delta_0} \times [-1,1]$. 
    \item[(iii)] $P_{V_0}(\gamma_{0\e}(0,{s}))=P_{V_0}(\gamma_{0\e}({p},{s}))$,   $P_{V_0}(\gamma_{0\e}({p},{-1}))>0$, $P_{V_0}(\gamma_{0\e}({p},{1}))<0$ for each $p\in   O^{ \delta_0}$, $s\in[-1,1]$.
    \item[(iv)] $\displaystyle \max_{(p,s)\in \partial (O^{ \delta_0} \times [-1,1])}\Gamma_\e (\gamma_{0\e}({p},{s})) < E_{V_0}-2\nu_3$.
    \item[(v)]   $c_\e\to E_{V_0} $  as $\e\to 0$.
  \end{description} 
\end{lemma}
\begin{proof}
  By Lemma \ref{lemma2.3}, it suffices to fix $\theta_1, \nu_3, \e_3>0$ sufficiently small to get the conclusion.
\end{proof}
\begin{proposition}\label{pro4.2}
  There is $\e_4\in (0, \min\set{\e_1, \e_2, \e_3})$ such that for $\e\in (0, \e_4)$, there is a sequence 
  $\{u_n\}\subset Z_\e(\rho_0, 3\delta_0)\cap \Gamma_\e^{c_\e}$ such that 
  \[
 \| \Gamma_\e'(u_n)\|_{H^{-1}_\e} \to 0, \quad \mbox{as } n\to +\infty.  
  \] 
\end{proposition}
\begin{proof}
Assume by contradiction, for a sequence  of $\e\to 0$, there is $\nu_\e>0$, such that
\[
\|\Gamma_\e'(u)\|_{H_\e^{-1}}\geq 2\nu_\e \mbox{ provided that } u\in    Z_\e(\rho_0, 3\delta_0)\cap \Gamma_\e^{c_\e}.
\] 
Together with Corollary \ref{cor3.2} and Proposition \ref{gradient estimate 2},
 we can find a locally lipschitzian pseudo gradient vector field $\mathcal{W}_\e : H_\e \to H_\e$ such that 
\begin{description}
  \item[(i)]  $\|\mathcal{W}_\e(u)\|_\e \leq 1$ and $\Gamma_\e'(u) \mathcal{W}_\e(u) \leq 0$ for $u\in Z_\e(\rho_0, 3\delta_0)$;
  \item[(ii)] $\mathcal{W}_\e(u)=0$ if $u\in H_\e\setminus Z_\e(\rho_0, 3\delta_0)$ or $u\in \Gamma_\e^{E_{V_0}-2\nu_3}$.
  \item[(ii)]  $\Gamma_\e'(u) \mathcal{W}_\e(u) \leq -\nu_\e$ for 
  $u\in Z_\e(\rho_0, 3\delta_0)\cap \Gamma_{\e, E_{V_0}-\nu_0}^{c_\e}$, where $\nu_0=\min\set{\nu_1, \nu_2, \nu_3}$ and
   \[
    \Gamma_{\e, E_{V_0}-\nu_0}^{c_\e} :=\set{u\in \Gamma_\e^{c_\e} |\Gamma_\e(u)\ge E_{V_0}-\nu_0}.\]  
  \item[(iii)] $\Gamma_\e'(u) \mathcal{W}_\e(u) \leq -\nu_0$ if $u\in (Z_\e(\frac23\rho_0, 2\delta_0)\setminus 
  Z_\e( \frac{1}{3}\rho_0, 2\delta_0))
  \cap \Gamma_{\e, E_{V_0}-\nu_0}^{c_\e}$.
  \item[(iv)] $\Gamma_\e'(u) \mathcal{W}_\e(u)\leq -\nu_0\e$ if $u\in (Z_\e(\frac23\rho_0, 2\delta_0)\setminus 
  Z_\e(\frac{1}{3}\rho_0, \delta_0))
  \cap \Gamma_{\e, E_{V_0}-\nu_0}^{c_\e}$.
\end{description}
We define 
\[
\begin{cases}\displaystyle
  \frac{\rd \eta(t, u)}{\rd t} =\mathcal V_\e(\eta(t,u))\\
  \eta(0, u)=u.
\end{cases}  
\]
For $t_* = \frac{\nu_0}{\nu_\e}$,
we consider $\gamma_\e(p,s)= \eta(t_*, \gamma_{0\e}(p,s))$, $(p, s) \in O^{\delta_0}\times [-1,1]$.
By Lemma \ref{lem4.1} (iv) and property (ii) of $\mathcal W_\e$,
 $\gamma_\e(p,s)= \gamma_{0\e}(p,s)$ for $(p, s)\in \partial (O^{\delta_0}\times [-1,1])$ if $\e\in(0, \e_3)$.
Consider the map $\mathcal F_\e: O^{\delta_0}\times [-1,1] \to \R^N\times \R$ defined by 
\[
\mathcal F_\e(p,s)= (\e \Upsilon(\gamma_\e(p,s)), P_{V_0}(\gamma_{\e}(p,s))).
\]
By Lemma \ref{lem4.1} (ii),
we have
$\e \Upsilon(\gamma_{0\e}(p,s))\to p$,
uniformly for $(p,s)\in O^{\delta_0}\times[-1,1]$ as 
$\e\to 0$. Then 
\begin{align*}
&\deg(\mathcal F_\e(p, s), O^{\delta_0}\times [-1,1], (0,   0))\\
=&  
\deg((\e \Upsilon(\gamma_{0\e}(p,s)), P_{V_0}(\gamma_{0\e}(p,s)) ), O^{\delta_0}\times [-1,1], (0,0))\\
=&\deg((id, P_{V_0}(\gamma_{0\e}(0,s)) ), O^{\delta_0}\times [-1,1], (0,0))\\
=&\deg((P_{V_0}(\gamma_{0\e}(0,s)), [-1,1], 0)=-1.
\end{align*} 
Then we conclude that 
there is $u_\e\in \gamma_\e(O^{\delta_0}\times [-1,1])$ such that $\e\Upsilon(u_\e) =0\in \mathcal V$ and $P_{V_0}(u_\e)=0$. 
Let $v_\e \in \gamma_{0\e}(O^{\delta_0}\times[-1,1])\subset Z_\e( \frac13 \rho_0, \delta_0)\cap \Gamma_\e^{c_\e}$ be  such that $u_\e = \eta(t_*, v_\e)$.
By property (i) of $\mathcal W_\e$, $\Gamma_\e(\eta(t, v_\e))$ is decreasing in $t$.
Hence, $\Gamma_\e(u_\e)\leq \Gamma_\e(v_\e) \le c_\e$.
By Lemma \ref{lem3.12}, we have 
\[
  \int_{\R^N\setminus B(0, 2 \e^{-\frac12})}  u_\e^2\le\int_{\R^N} \chi_{\e, u_\e} u_\e^2 \leq C\e^{\frac12}.  
\]
Then 
\begin{equation}\label{eq4.1}
  \begin{aligned}
    \Gamma_\e(u_\e) &\geq \frac12 \int_{\R^N} |\nabla u_\e|^2 +V_\e u_\e^2 -\int_{\R^N} F(u_\e)-\frac1N P_{V_0}(u_\e)\\
    &=\frac{1}N\int_{\R^N} |\nabla u_\e|^2 +\frac12\int_{\R^N} (V_\e-V_0) u_\e^2\\
    &\geq \frac{1}N\int_{\R^N} |\nabla u_\e|^2 -\frac12\int_{\R^N\setminus B(0, 2\e^{-\frac12})} V_0 u_\e^2+o_\e(1)\\
    &=L_{V_0}(u_\e) -\frac1N P_{V_0}(u_\e) +o_\e(1)\geq E_{V_0} +o_\e(1).
  \end{aligned}
\end{equation}
Hence for $\e$ small, $\eta([0,t_*], v_\e)\subset \Gamma_{\e, E_{V_0}-\nu_0}^{c_\e}$.
We will get contradictions in the following cases.

\noindent
{\bf Case 1.} $\eta([0,t_*], v_\e)\subset Z_\e(\frac23\rho_0, 2\delta_0)$. 
In this case, we use property (ii) of $\mathcal W_\e$ to get 
\begin{align*}
  \Gamma_\e(u_\e)=& \Gamma_\e(v_\e) +\int_{0}^{t_*} \frac{\rd}{\rd t} \Gamma_\e(\eta(t, v_\e)) \rd t\\
\le &c_\e +\int_0^{t_*} \Gamma_\e'(\eta(t, v_\e)) \mathcal W_\e(\eta(t, v_\e)) \rd t\\
\le &c_\e - t_*\nu_\e \leq c_\e -\nu_0.
\end{align*}

\noindent
{\bf Case 2.} $\eta(t, v_\e)\notin Z_\e( \frac23\rho_0, 2\delta_0)$ for some $t\in[0,t_*]$. There are two subcases.

\noindent
{\bf Subcase 2.1.} There are $t_1, t_2\in [0, t_*]$
with $t_1<t_2$ such that   $\eta([t_1, t_2], v_\e)\subset Z_\e(\frac23\rho_0, 2\delta_0)\setminus 
Z_\e( \frac{1}{3}\rho_0, 2\delta_0)$ and
$\|\eta(t_1, v_\e)-\eta(t_2, v_\e)\|\geq \frac13\rho_0$. By property (i) of $\mathcal W_\e$, we have 
$t_2-t_1\geq \frac13\rho_0$. By property (iii) of $\mathcal W_\e$, we have 
\begin{align*}
  \Gamma_\e(u_\e)\le & \Gamma_\e(v_\e) +\int_{t_1}^{t_2} \frac{\rd}{\rd t} \Gamma_\e(\eta(t, v_\e)) \rd t\\
\le &c_\e +\int_{t_1}^{t_2} \Gamma_\e'(\eta(t, v_\e)) \mathcal W_\e(\eta(t, v_\e)) \rd t\\
\le &c_\e - (t_2 -t_1)\nu_0 \leq c_\e -\frac13\rho_0\nu_0.
\end{align*}

\noindent
{\bf Subcase 2.2.} There are $t_1, t_2\in [0, t_*]$
with $t_1<t_2$ such that   $\eta([t_1, t_2], v_\e)\subset Z_\e(\frac23\rho_0, 2\delta_0)\setminus 
Z_\e( \frac{1}{3}\rho_0,  \delta_0)$ and
$|\Upsilon(\eta(t_1, v_\e))-\Upsilon(\eta(t_2, v_\e))|\geq \frac13\rho_0\e^{-1}$. Then by Lemma \ref{lem3.4} (ii)
and  property (i) of $\mathcal W_\e$,
\[
|t_1-t_2|\geq \frac{1}{3}\rho_0 \e^{-1} D_1^{-1}.  
\]
By property (iv) of $\mathcal W_\e$, we have 
\begin{align*}
  \Gamma_\e(u_\e)\le & \Gamma_\e(v_\e) +\int_{t_1}^{t_2} \frac{\rd}{\rd t} \Gamma_\e(\eta(t, v_\e)) \rd t\\
\le &c_\e +\int_{t_1}^{t_2} \Gamma_\e'(\eta(t, v_\e)) \mathcal W_\e(\eta(t, v_\e)) \rd t\\
\le &c_\e - (t_2 -t_1)\nu_0 \e \leq c_\e -\frac13\rho_0\nu_0D_1^{-1}.
\end{align*}
In either case, taking limits  as $\e \to 0$, we get a contradiction to \eqref{eq4.1}.
\end{proof}

The existence of a critical point follows from the following compactness result.
\begin{proposition}\label{pro4.3}
  There is $\e_5\in (0, \e_3)$ such that for each $\e\in (0, \e_5)$, the (PS) sequence $\{u_n\}$ given in Proposition \ref{pro4.2} has a strong convergent subsequence.
\end{proposition}
The proof of Proposition \ref{pro4.3} is standard and will be given in Appendix. Now we complete the proof of Theorem \ref{th1.1}.
\begin{proof}[\bf Completion of proof of Theorem \ref{th1.1}]
By (V2) and (V3),  we  choose a sequence of open sets $\set{O_k}_{k=0}^\infty$ with $O_{k+1}\subset O_{k}\subset O$, $\cap_{k=0}^\infty O_k=\mathcal V$, and each $O_k$ satisfying 
\eqref{equ2.2} and \eqref{equ2.3}. 
By Proposition \ref{pro4.2} and Proposition \ref{pro4.3}, there are decreasing $\e_k\to 0$, $\delta_k\to 0$ such that $\Gamma_\e$ has  a critical point 
$u_{k,\e}$
in 
$Z_\e(\rho_0, 3\delta_j)\cap \Gamma_\e^{E_{V_0}+{c_{k,\e}}}$ for $\e\in (0, \e_j)$,
where $c_{k,\e}$ is decreasing with respect to $k$ and $c_{k,\e}\to E_{V_0}$ as $\e\to0$.
Define 
\[ u_\e = u_{k,\e} \ \mbox{ for } \e\in [\e_{k+1}, \e_k).
  \]
  Then for   $\e\in (0, \e_0)$ we have found a critical point $u_\e$ of $\Gamma_\e$.
  By Proposition \ref{lemdecay}, we know that $u_\e$ is in fact a solution of the original problem \eqref{eq1.2}.
  Since for any sequence $\e \to 0$, $u_\e$ satisfies the assumptions of Proposition \ref{concompact}, 
  we know that, up to a subsequence,
  there exist
$U\in S_0$ and 
$ z_{\e} \in \frac{1}{\varepsilon} \Omega$ such that as $\e\to 0$  
\begin{equation*}
  \|u_{\e}-   (\phi_{\e }U)  (\cdot-z_{\e} )\|_{ {\varepsilon }}\to 0.
\end{equation*}
Since $\dist (\e\Upsilon (u_\e), \mathcal V)\to 0$ and $|\Upsilon(u_\e)-z_\e|\leq 2R_0$, we have 
$\dist(\e z_\e, \mathcal V)\to 0$.
\end{proof}
\section{Appendix}
\begin{proof}[{\bf Proof of Proposition \ref{concompact}}]
  Let  $\varepsilon_n, u_n$
   satisfy 
 \eqref{13}. 
 By the compactness of $S_0$, we can write
 \begin{equation}\label{equa 4.1}
 u_n(x)= (\phi_{\e_n}\tilde U)   (x-y_n )+w_n(x),\quad \|w_n\|_{ \varepsilon_n}\leq \rho_0,\quad \varepsilon_n\Upsilon (u_n)\in O^{3\delta_0}, 
 \end{equation}
  where $y_n\in \frac{1}{\e_n}\Omega$, $\tilde U\in S_0$. By Lemma \ref{lem3.4} (i), $|y_n-\Upsilon(u_n)|\leq 2R_0$ and
   $\dist(\varepsilon_n y_n, O^{3\delta_0}) \leq 2R_0\varepsilon_n\to 0.$ 
   Hence, by Lemma \ref{lem3.12} and Proposition \ref{lemdecay}, for some $C, c>0$ independent of   $n$ and any $R>0$, there hold 
 \begin{gather}
   \|u_n\|_{\varepsilon_n},\ \int_{\R^N} f(u_n)u_n ,\ \int_{\R^N}F(u_n), \Phi_{\e_n}(u_n) \leq C  \notag\\
    \int_{\R^N\setminus   B(y_n,R) } \left(|\nabla u_n|^2+  u_n^2 \right) \mathrm dx\leq C e^{-cR}+o_n(1).\label{eq39}
 \end{gather}
 
 Up to a subsequence, we assume  
 $\varepsilon_n y_n\to y_0\in {O^{3\delta_0}}$ and $ u_n(\cdot+y_n)\rightharpoonup W \neq 0$ in $H^1(\R^N)$.
 By Lemma \ref{Phi'}, \eqref{eq39},   there holds
 \begin{equation}\label{eq3.4}
   \Phi_{\varepsilon_n}'(u_n)v-4\Phi_{\varepsilon_n}(u_n)^{\frac12}\e^{-\frac12}\int_{\R^{N} }\chi_{\e_n,u_n} u_nv =o_n(1)\|v\|_{\e_n},\ \ v\in H_{\varepsilon_n}.
   \end{equation}
For any $\varphi\in C_0^{\infty}(\R^N)$, setting $v=\varphi(\cdot-y_n)$,
we can verify that 
 $W$ satisfies
 $$-\Delta W + V(y_0)W=f(W) \quad \text{in}\  \R^N.
 $$

 \noindent\textbf{Step 1.} 
 Setting $v_n:=u_{n}-  (\phi_{\e_n} W)  (\cdot-y_{n}  )$, we show $|v_n|_p\to0$ for $p\in(2,2^*)$.
  
 Otherwise, by Lions' Lemma,
 there is $y_n'$ such that $|y_n'-y_n |\to\infty$  
 and $\limsup_{n\to\infty}\|u_n(\cdot+y_n')\|_{L^2(B(0,1))}>0$.
Assume $u_n(\cdot+ y_n')\rightharpoonup u_1$ in $H^1(\R^N)$ for $u_1\neq 0$.
For each $R>0$, let   $\eta_R\in C^\infty_0(\R^N,[0,1])$  be such that  $\eta_R=1$ in $B(y_n',R)$, $\eta_R=0$ in $\R^N\setminus B(y_n',2R)$ and $|\nabla \eta_R|\leq 2/R $.
 We have $\supp (\eta_R u_n) \subset \R^N \setminus B(\Upsilon(u_n), 4R_0)$ for large $n$.
 Then by Lemma \ref{Phi'},
 \[
  \Phi_{\varepsilon_n}'(u_n)(\eta_R u_n)= 4\Phi_{\varepsilon_n}(u_n)^{\frac12}\e_n^{-\frac12} \int_{\R^{N} }\chi_{\e_n,u_n}       \eta_R u_n^2 \mathrm dx \ge 0.
 \]
 Therefore,
 $$\begin{aligned}
   o_n(1)=&\Gamma_{\varepsilon_n}'(u_n)(\eta_R u_n) \\
\ge&\int_{\R^N}\left(\nabla u_n\nabla (\eta_R u_n) +  V_\varepsilon \eta_R u_n^2-\eta_R f(u_n) u_n \right) \mathrm d x 
   \\
 \geq & \int_{\R^N}\eta_R\left(|\nabla u_n|^2  +    u_n^2-  f(u_n) u_n \right) \mathrm d x - \frac{2}R \int_{\R^N} |u_n \nabla u_n|   \mathrm d x.
 \end{aligned}
 $$
 Since $u_n(\cdot+ y_n')\rightharpoonup u_1$, by compact embedding we have 
 \[
  \int_{\R^N} \eta_R f(u_n) u_n \to \int_{\R^N} \eta_R f(u_1) u_1.
 \]
 By Fatou's Lemma, we have 
 \[\liminf_{n\to +\infty}\int_{\R^N}\eta_R\left(|\nabla u_n|^2  +    u_n^2\right) \geq \int_{\R^N}\eta_R\left(|\nabla u_1|^2  +    u_1^2\right)
  \]
  Therefore, taking limits as $n\to+\infty$ and by the arbitrary choice of $R$, we obtain 
  \[
  \|u_1\|^2\leq \int_{\R^N} f(u_1)u_1.  
  \]
 
  By Lemma \ref{lemma26},  $|u_1|_2\geq \rho_2>\rho_0$
 which contradicts with \[|u_1|_2= \lim_{R\to+\infty}\lim_{n\to+\infty}\|u_n\|_{L^2(B(y_n', R))}= \lim_{R\to+\infty}\lim_{n\to+\infty}\|w_n\|_{L^2(B(y_n', R))}\leq \rho_0.\]
 
 \noindent{\bf Step 2.} $\| v_n\|_{\e_n}\to 0.$
 
 We test 
 \eqref{13} by $v_n$ and use \eqref{eq3.4} to get
 \begin{equation}\label{38'}
    \begin{aligned}
   (u_n, v_n)_{\e_n} - \int_{\R^N} f(u_n) v_n  
    +4\Phi_{ \varepsilon_n}(u_n)^{\frac12}\e_n^{-\frac12}\int_{\R^{N} }\chi_{\e_n,u_n} u_n v_n \mathrm dx
  =o_n(1).
   \end{aligned}
 \end{equation}
 By \eqref{eq39},
 $$\begin{aligned}
  4\Phi_{ \varepsilon_n}(u_n)^{\frac12}\e_n^{-\frac12}\int_{\R^{N} }\chi_{\e_n,u_n}  u_n v_n  \mathrm dx 
   \ge&-
   4\Phi_{ \varepsilon_n}(u_n)^{\frac12}\e_n^{-\frac12}\int_{\R^{N} \setminus B(y_n,  (2\e_n)^{-1/2}) }    u_n   W (\cdot-y_{n} )dx\\
 \geq& -C \e_n^{-\frac12}e^{-\frac{C}{\sqrt\e_n}}=o_n(1).
 \end{aligned}
 $$
 Hence, $  (u_n, v_n)_{\e_n} \le \int_{\R^N} f(u_n) v_n + o_n(1)$.
 Then  we have
 $$\begin{aligned}
   \|v_n\|^2_{ {\varepsilon_n}} 
   =\int_{\R^N} \left(\nabla    (\phi_{\e_n}W)  (\cdot-y_{n} )\nabla v_n + V_\varepsilon   \phi_{\e_n}  W  (\cdot-y_{n} )v_n\right)+\int_{\R^N} f(u_n)  v_n  +o_n(1).
 \end{aligned}
 $$
 We have, by $u_n(\cdot+y_n)-W\rightharpoonup 0$ in $H^1(\R^N)$  and the decay property of $W$,
 \[
  \int_{\R^N}  \nabla    (\phi_{\e_n}W)  (\cdot-y_{n} )\nabla v_n =\int_{\R^N}  \nabla    W \nabla (u_n(\cdot+y_n)-W) +o_n(1)=o_n(1),
 \]
 \[
  \int_{\R^N}|V_\varepsilon    (\phi_{\e_n}  W)  (\cdot-y_{n} )v_n|\leq \int_{\R^N} V_0 |W (u_n(\cdot+y_n)-W)|\ =o_n(1).
 \]
 By (F1) and \eqref{eq2.11},  for each $\tau>0$, there is $C_\tau>0$  such that $ f(t) \le \tau t+ C t^p$, for $p\in (2,2^*)$.  
  Hence, by Step 1, $\limsup_{n\to+\infty}\int_{\R^N} f(u_n) v_n\leq C\tau$ , and by the arbitrary choice of $\tau$, there holds $\lim_{n\to\infty} \|v_n\|^2_{ {\varepsilon_n}}=0$.
 
 \noindent\textbf{Step 3.} Completion of the proof.
 Let $z$ be the unique maximum point of $W$, 
 since \[\int_{\R^N\setminus B(0, 2R_0)}W^2=\lim_{n\to\infty}\int_{\R^N\setminus B(y_n, 2R_0)}u_n^2\leq \frac{\rho_1^2}{16}, \] 
 we have $|z|\leq 2R_0$.
 We set $U=W(\cdot+z)\in H^1(\R^N)$.
 By Step 2,
  we have 
 \[
 \lim_{n\to+\infty}\int_{\R^N}F(u_n)= \int_{\R^N} F(W )= \int_{\R^N} F(U ).
 \]
 Therefore,  
 $$\begin{aligned}
   L_{V(y_0)}(U)\leq &\lim_{n\to\infty} \Gamma_{\e_n}(u_n) \leq \ell E_{V_0}.
 \end{aligned}
 $$
 Then $U\in S_0$. Setting $z_{n}=y_n+z$, we have completed the proof.
 \end{proof}
\begin{proof}[\bf Proof of Proposition \ref{pro4.3}]
  Let $\{u_n\} $ be the (PS) sequence of $\Gamma_{\varepsilon}$ obtained in Proposition \ref{pro4.2} for given small $\varepsilon>0$. Clearly, $\|u_n\|_\varepsilon$ and $\Phi_\varepsilon(u_n)$ are bounded by a constant independent of $\varepsilon$.
  We have 
\begin{equation}\label{une}
\int_{\mathbb R^N\setminus B(\Upsilon(u_n), 2 \e^{-\frac12})}  u_n^2\leq C \e^{\frac12}.
\end{equation}
Moreover, $\Upsilon(u_n) \in \frac{1}{\e} O^{3\delta_0}$ is bounded for fixed $\e$.
Passing to a subsequence if possible, we may assume that there exists $u_{\e}\in H_{\e}$ such that
$u_n\rightharpoonup u_{\e}$ in $H_{\e}$, $u_n\rightarrow u_{\e}$ in $L^p_{loc}(\mathbb{R}^N), \forall
p\in[1,2^*), u_n(x)\rightarrow u_{\e}(x)$ for a.e. $x\in\mathbb{R}^N$. 
We next claim that $u_n\to u_\e$ in $L^p(\R^N)$ for any $p\in (2,2^*)$.
In fact, if not, by Lions' Lemma \cite{lions}, we may assume that 
$\lim\limits_{n\rightarrow+\infty}\sup\limits_{x\in\mathbb{R}^N}\int_{B (x, 1)}|u_n-u_{\e}|^2>0$.
Then there exists $\{y_n\} \subset \mathbb{R}^N$
 such that $\liminf\limits_{n\rightarrow+\infty}\int_{B (y_n, 1)}|u_n-u_{\e}|^2>0$. Clearly, $|y_n|\to+\infty$.
  Hence we have  
 \[
  \liminf_{n\rightarrow+\infty}\int_{B(y_n, 1)} u_n^2 >0.
 \]
 Since $u_n(\cdot+y_n)$ is bounded in $H^1(\R^N)$, we assume up to a subsequence that $u_n(\cdot+y_n)\rightharpoonup v$ in $H^1(\R^N)$ for some 
 $v\in H^1(\R^N)\setminus \{0\}$. By \eqref{une},  
 \begin{equation}\label{eqv}
  |v|_2^2\leq C\e^{\frac12}.
 \end{equation}
 For any $R>1$, take $\xi_{n,R}\in C^\infty(\R^N; [0,1])$ such that $\xi_{n,R}=1$ in $B(y_n, R)$, $\xi_{n,R}=0$ in $  B(y_n, 2R)$, and $|\nabla \xi_{n,R}|\leq 2/R$. Then $\{\xi_{n,R}u_n\}$ is bounded in $H_{\e}$ by a constant independent of $n, R$. Moreover,  for large $n$, $\supp(\xi_n u_n)\subset \R^N\setminus B(\Upsilon(u_n), 2R_0)$. We have 
\[
\Phi_\e'(u_n)(\xi_{n,R} u_n) =   4\Phi_{\varepsilon}(u_n)^{\frac12} \e^{-\frac12}\int_{\R^{N} }\chi_{\e, u_n} \xi_{n,R} u_n^2\geq 0.
\]
Hence, 
\[
o_n(1)= \Gamma_\e'(u_n) (\xi_{n, R} u_n ) \geq   \int_{\R^N} \nabla u_n \nabla (\xi_{n,R} u_n) +V_\e\xi_{n,R} u_n^2 -\int_{\R^N} f(u_n)\xi_{n,R} u_n.
\]
We have 
\[
  \liminf_{n\to+\infty}\int_{\R^N} \nabla u_n \nabla (\xi_{n,R} u_n) +V_\e\xi_{n,R} u_n^2 \geq   \int_{B(0,R)} |\nabla v|^2+v^2 -\frac{C}R,
\]
\[
  \limsup_{n\to+\infty} \int_{\R^N} f(u_n)\xi_{n,R} u_n \leq  \limsup_{n\to+\infty} \int_{B(y_n, 2 R)} |f(u_n)  u_n|= \int_{B(0, 2 R)} |f(v)  v|.
\]
By the arbitrary choice of $R$, we obtain 
\[
\|v\|^2\leq \int_{\R^N}|f(v)v|.  
\]
Hence, by Lemma \ref{lemma26}, we have $|v|_2\geq \rho_2$. Then we can get a contradiction to \eqref{eqv} for each $\e\in(0, \e_5)$
 provided that $\e_5$ is fixed sufficiently small.
\end{proof}

\section*{Acknowledgements}
The author expresses sincere gratitude to Prof. Kazunaga Tanaka, Prof. Xu Zhang, and Dr. Haoyu Li for their invaluable  suggestions and   comments.

\end{document}